\documentclass[10pt, twoside, fleqn, a4paper]{article}

\usepackage{latexsym}  
\usepackage{amscd}    
\usepackage{theorem}  
\usepackage{mathbbol} 
\usepackage{amsfonts}  
\usepackage{xspace}  
\usepackage{amssymb} 
\usepackage{fancyhdr} 
\usepackage{mathrsfs} 
\usepackage{amsmath} 
\usepackage{graphics} 
\usepackage{graphicx} 
\usepackage{euscript}
\usepackage{psfrag}
\usepackage{empheq} 
\usepackage{mathabx} 
\usepackage[parfill]{parskip}     
\usepackage[colorlinks=true, allcolors=blue]{hyperref}
\usepackage{cancel} 
\usepackage{bigints}

\title{\large A VIEW TOWARDS MIXING IN HOLOMORPHIC CORRESPONDENCES}
\author{Sathi Trikkadeeri Mana\footnote{Indian Institute of Science Education and Research Thiruvananthapuram (IISER-TVM). \\ ORCID: 0009-0000-1444-4604 \ \ email: \texttt{sathitm23@iisertvm.ac.in}} \  and \  Bharath Krishna Seshadri\footnote{Indian Institute of Science Education and Research Thiruvananthapuram (IISER-TVM). \\ ORCID: 0009 - 0004 - 0541 - 8040\ \ email: \texttt{bharathmaths21@iisertvm.ac.in} \ (Corresponding Author)}}

\DeclareFontFamily{OT1}{pzc}{}
\DeclareFontShape{OT1}{pzc}{m}{it}%
              {<-> s * [0.900] pzcmi7t}{}
\DeclareMathAlphabet{\mathpzc}{OT1}{pzc}%
                                 {m}{it}

{\theorembodyfont{\slshape} \newtheorem{theorem}{Theorem}[section]}
{\theorembodyfont{\slshape} \newtheorem{definition}[theorem]{Definition}}
{\theorembodyfont{\slshape} \newtheorem{lemma}[theorem]{Lemma}}
{\theorembodyfont{\slshape} \newtheorem{proposition}[theorem]{Proposition}}
{\theorembodyfont{\slshape} }
{\theorembodyfont{\slshape} } 
{\theorembodyfont{\slshape} } 
{\theorembodyfont{\slshape} \newtheorem{remark}[theorem]{Remark}}
{\theorembodyfont{\slshape} \newtheorem{example}[theorem]{Example}}

\numberwithin{equation}{section}

\newenvironment{proof}{\paragraph{Proof:}}{\hfill$\blacklozenge$}

\usepackage[backend=biber,style=alphabetic]{biblatex}
\begin{document}

\maketitle

\begin{abstract} 
\noindent 
In this manuscript we develop a theory of mixing and weakly mixing in the study of dynamics of holomorphic correspondences defined on a compact connected complex manifold. We also connect these notions to the theory of ergodicity of holomorphic correspondences developed by Londhe. Further, we give motivation and illustrative examples that compare the present scenario with that of maps. Finally, we study product of two holomorphic correspondences and use them to characterise weakly mixing. 
\end{abstract}

\begin{tabular}{l l} 
{\bf Keywords} & Holomorphic correspondences \\ 
& Mixing and weakly mixing \\ 
& Ergodicity \\ 
& Koopman operator \\
& Product correspondences \\
\\
{\bf MSC Subject} & \\ 
{\bf Classifications} & Primary : 37A05, 37A25 \\
& Secondary : 32A08, 32H50, 37A50 \\ 
\end{tabular} 

\bigskip 

\section{Introduction} 

The study of complex dynamics mainly involves understanding the dynamics of both single - valued and multi - valued functions defined on spaces with complex analytic structure. The typical single - valued functions studied are the rational maps on the Riemann sphere. Some ideas generalising the theory of iterations of a single rational map include the study of Kleinian groups and that of rational semigroups. However, in any case, each of the iterate in these cases remain single - valued functions. With regards to the multi - valued functions in case of complex dynamics, the prime examples are what are called holomorphic correspondences. This study mainly gained prominence starting with the works of Bullett and Penrose, like for example \cite{bullet_quad, bullpen_gallery, bulllpen_limit}. Various directions of study have been carried out in this domain since then. One of the modes of investigation is to study the measure theoretic and ergodic theoretic properties related to the dynamics of holomorphic correspondences and compare them with analogous results in the case of rational maps.

In the case of rational maps on $\widehat{\mathbb{C}}$, the works of Lyubich and that of the trio of Freire, Lopes and Ma\~{n}\'{e} proved the equidistribution result related to the backward images of a non exceptional point in $\widehat{\mathbb{C}}$ under the iterations of a given rational map. In \cite{lyubich_equi}, the author considered sequences of measures related to backward orbits and periodic points respectively and constructed the corresponding limits obtained via weak* convergence and explored the properties of the limits. The limit measure found in above results is found to have some nice ergodic theoretic properties, as the results in \cite{lyubich_equi} and \cite{flm_equi} indicate. 

The above ideas were first studied in the case of holomorphic correspondences by Dinh and Sibony. In fact, in the work \cite{dinhsibony_og}, they proved the analogue of equidistribution of backward orbits for what are called meromorphic correspondences defined on K\"{a}hler manifolds, which actually subsumes the holomorphic correspondences and in \cite{dinh_poly}, Dinh had investigated about polynomial correspondences. Similar to the case of maps, the limit measure obtained in this case also turns out to have some nice properties. For example, the recent work \cite{dinh_kauffmann} and the paper \cite{dinh_poly} mentions about the concept of exponential mixing and proves that the limit measure in the appropriate equidistribution theorems related to the holomorphic correspondences considered in those works, have that property. It must be noted that the notions of invariance of measure under the correspondence, the idea of iteration, etc come in a different avatar - understandably so due to the multi - valued nature of a correspondence. 

These limiting measures serve as prototypes for studying various ergodic theoretic properties and further motivate us to generalise the notions in the ergodic theory of maps to the setting of holomorphic correspondences. Given a holomorphic correspondence $F$ on a compact connected complex manifold $X$ (see Section \ref{sec:two} for its definition), the main objects of interest from a measure theoretic perspective is what are called $F^{*}$ invariant measures, which we shall introduce in Section \ref{sec:two}. These are the analogues of invariant measures found in the case of maps. In the works \cite{londhe_recurrence} and \cite{londhe_ergodicity}, Londhe took this theory forward by introducing the notions of recurrence and ergodicity of a holomorphic correspondence $F$ with respect to an $F^{*}$ invariant measure and proved the analogues of the Poincare recurrence theorem and the Birkhoff ergodic theorem in this setting. 

Let $F$ be a holomorphic correspondence on $X$ and let $\mu$ be an $F^{*}$ invariant measure. Motivated by the above mentioned results, in this work, we investigate the mixing notions in the study of holomorphic correspondences. We now give a brief outline on the contents of our paper. In Section \ref{sec:two}, we give an overview of the existing results on the dynamics of holomorphic correspondences, that we shall use in our paper. We further introduce the Koopman operator associated to a holomorphic correspondence. Section \ref{sec:three} serves as a launchpad for the main part of this paper by stating the definitions of $F$ being mixing (resp. weakly mixing) with respect to $\mu$ (see Definition \ref{defn:mixing}), that we shall work with in this article. We also give a thorough motivation for our definition and a heuristic justification on our choice of definition and how there is a marked difference in the case of mixing phenomenon in the study of rational maps (or for that matter any continuous map on a compact space) and that in the study of holomorphic correspondences. Section \ref{sec:four} contains various results relating the mixing phenomenon in holomorphic correspondences with that of ergodicity and also characterises ergodicity by showing that it can be viewed as ``mixing on the average" (refer Theorem \ref{thm:ergequiv}). We also state a characterising result for mixing and weakly mixing. Further, this section provides an illustrative example, namely Example \ref{example} (see the discussion immediately succeeding it as well) and various observations (see in particular Propositions \ref{thm:ineqeg} and \ref{prop:aiislim}) comparing how interaction between mixing and ergodicity differs between the case of maps and that of correspondences. In Section \ref{sec:five}, we discuss the dynamics of product of two holomorphic correspondences $F_1$ and $F_2$ defined on compact connected complex manifolds $X_1$ and $X_2$ respectively. In Section \ref{sec:six}, we provide an application for the tools developed in Section \ref{sec:five} to prove a characterising result regarding the ergodicity of product correspondences, which turns out to be related to the notion of weakly mixing. We conclude the paper with Theorem \ref{thm:main}, the main result of Section \ref{sec:six}, that captures this idea. 

\section{An invitation to holomorphic correspondences}
\label{sec:two}

We give a brief overview on the basics of holomorphic correspondences that we use in our work. An interested reader can also check the papers \cite{dinhsibony_og}, \cite{bs_p1}, \cite{bs_entropy} and \cite{londhe_recurrence} for further details regarding the same. The presentation given in this section mostly follows the one given in \cite{londhe_ergodicity}. In this paper, unless stated otherwise, $X$ shall denote a compact connected complex manifold of dimension $k$. Further, we consider our manifolds to be Hausdorff. Let $\pi_{1}$ and $\pi_{2}$ be the projections defined from $X \times X$ onto $X$, given by  $\pi_{1} (z, w) = z \ \  \text{and} \ \  \pi_{2} (z,w) = w.$
\begin{definition} 
\label{defn:holcorresp}
A holomorphic correspondence on $X$, denoted by $\Gamma$ is a formal sum given by 
\begin{equation} 
\label{eqn:holocorr} 
\Gamma\ \ =\ \ \sum_{j\, =\, 1}^{N} m_{j} \Gamma_{j}, 
\end{equation} 
where $m_{j} \in \mathbb{Z}_{+}$ for every $1 \le j \le N$ (denotes the multiplicity of the sub-variety $\Gamma_{j}$) and $\Gamma_j \, \text{'s}$ are distinct irreducible complex subvarieties of $X \times X$ of dimension $k$, that satisfies the following conditions: 
\begin{enumerate} 
\item For each $1 \le j \le N$, we have $\left. \pi_{1} \right|_{\Gamma_{j}}$ and $\left. \pi_{2} \right|_{\Gamma_{j}}$ are surjective; 
\item The set $\bigcup\limits_{j\, =\, 1}^{N} \pi_{2} \left( \pi_{1}^{-1} \{z_{0}\} \cap \Gamma_{j} \right)$ and $\bigcup\limits_{j\,=\, 1}^{N} \pi_{1} \left( \pi_{2}^{-1} \{z_{0}\} \cap \Gamma_{j} \right)$, where the points are repeated according to their multiplicities, is finite for all $z_{0} \in X$.
\end{enumerate} 
\end{definition}

Throughout this paper, there are instances where we use the term ``set" to represent the collection of all images and preimages of points in $X$, counted with multiplicities. In this regard, we denote by $F$ the set valued map induced by the variety representation $\Gamma$. We shall henceforth use $\Gamma$ and $F$ interchangeably as per the demands of the situation. In any case both refer to the same object, the holomorphic correspondence under consideration.

\subsection{The dynamics}
So far we have introduced the object that we are going to work with. In order to do dynamics, it is important to define the forward and backward images and how iteration works in this system. For a holomorphic correspondence $\Gamma =\displaystyle  \sum_{j\, =\, 1}^{N} m_{j} \Gamma_{j}$, let $\left | \Gamma \right|$ denote the set $\bigcup \limits _{j=1}^{N} \Gamma_{j}$. Then, for a subset $A$ of $X$, we define
$F(A) := \pi_{2} (\pi_{1}^{-1} (A) \cap \left | \Gamma \right |) \ \ \text{and} \ \ F^{\dagger}(A) := \pi_{1} (\pi_{2}^{-1} (A) \cap \left | \Gamma \right |).$ We call $F(A)$ and $F^{\dagger} (A)$ respectively the image and preimage of the set $A$ under the correspondence $\Gamma$. When $A = \{z\}$, where $z \in X$, we simply denote the above two sets by $F(z)$ and $F^{\dagger} (z)$ respectively. Now, we define $F^{n} (A)$ and $(F^{n})^{\dagger} (A)$ recursively by setting $F^n(A) := F(F^{n - 1} (A))$ and $(F^{n})^{\dagger} (A) = F^{\dagger} ((F^{n - 1})^{\dagger} (A))$. From \cite{londhe_ergodicity}, it is known that, for any $n \in \mathbb{N}$, the sets $F^n(B)$ and $(F^{n})^{\dagger} (B)$ are Borel whenever $B$ is a Borel set in $X$. The cardinality of the set $F(z)$ for a generic point $z \in X$ is called the \emph{forward degree} of $F$ and is denoted by $d_{f}$, while the cardinality of the set $F^{\dagger} (z)$ (where $z$ is a generic point in $X$) is called the \emph{topological degree} of $F$ and is denoted by $d$. A remark from \cite{londhe_recurrence} (see Section 2.1 therein) comments about the topological degree of $\Gamma$ and also sheds light on the generic points. 

Given two holomorphic correspondences on $X$ whose variety representations are $\Gamma_1$ and $\Gamma_2$, one can define the composition $\Gamma_2 \circ \Gamma_1$. The details of the same can be found in Section 2.1 of \cite{bs_p1}. With this definition, one can show that $\Gamma^{\circ n}$, which is the $n-$ fold composition of the holomorphic correspondence $\Gamma$ turns out to be the variety representation corresponding to the set valued map $F^{n}$. Thus, for each $n \in \mathbb{N}$, the $n^{th}$ iterate of $F$, namely $F^n$ is also a holomorphic correspondence as it satisfies Definition \ref{defn:holcorresp}.  

\subsection{Measure theoretic aspects}

We now proceed to studying Borel probability measures on $X$ and how they interact with the dynamics of $F$. Analogous to the measure preserving criterion (which is defined using a push forward construction) that is studied in the case of maps,  $F^{*}$ invariance of measures is studied in the case of correspondences. This is obtained by a pull back construction and is typically done using the theory of currents. We denote the Borel sigma algebra of $X$ by $\mathscr{B} (X)$ and $\mathcal{M} (X)$ denotes the space of all Borel probability measures on $X$. Further, the space of continuous functions from $X$ to $\mathbb{R}$ is denoted by $\mathcal{C} (X, \mathbb{R})$. 

 Let $F$ be a holomorphic correspondence on $X$. For $\mu \in \mathcal{M} (X)$, consider the pull back $F^{*}\mu$ given by,
 \[
 \left \langle F^{*}\mu, \phi  \right \rangle =\displaystyle \bigintsss \limits_{X} \, \,  \displaystyle \sideset{}{'} \sum_{w \, \in \, F^{\dagger} (z)} \phi(w) \, \mathrm{d}\mu,
 \]
where $\phi \in \mathcal{C} (X, \mathbb{R})$. Here, the notation $\sideset{}{'} \sum$ is used to highlight the fact that the points $w$ are repeated according to multiplicity in the summation.

\begin{definition}
\label{defn:fstar}
    Let $F$ be a holomorphic correspondence on $X$. A measure $\mu \in \mathcal{M} (X)$ is said to be $F^{*}$ invariant if $F^{*}\mu = d\mu$.
\end{definition}

When viewed using integrals, the notion of $F^{*}$ invariance translates to the following equality.
\begin{equation}
\label{eqn:inteq}
\bigintsss \limits_{X} \, \,  \displaystyle \sideset{}{'} \sum_{w \, \in \, F^{\dagger} (z)} \phi(w) \, \mathrm{d}\mu = d \int \limits_{X} \phi \, \mathrm{d}\mu, \hspace{2mm} \text{for any} \hspace{2mm} \phi \in \mathcal{C} (X, \mathbb{R}).
\end{equation}

In fact, since $\mathcal{C} (X, \mathbb{R})$ is dense in $L^{q} (\mu)$ for $1 \le q < \infty$, Equation \ref{eqn:inteq} holds for all $\phi \in L^{q} (\mu)$ for all $1 \le q < \infty$. Henceforth, whenever we refer to Equation \ref{eqn:inteq} in this paper, we mean to use this equality for the function $\phi$ from the appropriate space as they appear in that context.

We now state a landmark result due to Dinh and Sibony, regarding the equidistribution of preimages of a point under a holomorphic correspondence. It must be noted that in the paper \cite{dinhsibony_og}, they have proved a vastly general result and we state it in the form that is relevant to our setting. 

\begin{theorem} [Dinh - Sibony, \cite{dinhsibony_og}]
\label{thm:ds}
Let $F$ be a holomorphic correspondence on a compact K\"{a}hler manifold $X$ of dimension $k$ with $d > d_{k - 1}$, where $d_{k - 1}$ is the dynamical degree of $F$ of order $k - 1$ and $d$ is the topological degree of $F$. Then, there is a pluripolar set $\mathcal{E} \subsetneq X$ such that 
\[ \frac{1}{d^{n}} \left( F^{n} \right)^{*} \delta_{z}\ \ \to\ \ \mu_{F},\ \ \ \ \text{in the weak* topology,\ for all}\ z \in X \setminus \mathcal{E}. \] 
Moreover, the measure $\mu_F$ is $F^{*}$ invariant and $\mu_F(\mathcal{E}) = 0$. Here, $\delta_{z}$ denotes the Dirac delta measure at the point $z$. 
\end{theorem} 
We refer to the measure $\mu_F$ in the above theorem as the Dinh - Sibony measure associated with $F$. One of the main aims of developing these measure theoretic techniques is to study the notion of ergodicity, which was extensively studied in the theory of holomorphic correspondences by Londhe in \cite{londhe_ergodicity}. We shall now state some definitions and results from the same work that shall be useful to us. 

\begin{definition}
    Let $\mu \in \mathcal{M} (X)$ be $F^{*}$ invariant with respect to a holomorphic correspondence $F$ on $X$. A set $B \in \mathscr{B} (X)$ is said to be almost invariant with respect to $F$ and $\mu$ if there exists $B' \in \mathscr{B} (X)$ with $B' \subseteq B$ and $\mu (B) = \mu (B')$ such that $F^{\dagger} (B') \subseteq B$. 
\end{definition}

It is known that if $B \in \mathscr{B} (X)$ is almost invariant with respect to $F$ and $\mu$, then so is $B^{c}$. Further for any $n \in \mathbb{N}$, $B$ is almost invariant with respect to $F^{n}$ and $\mu$ (see Lemma 3.1 in \cite{londhe_ergodicity} for a proof of the same).

\begin{definition}
    Let $F$ be a holomorphic correspondence on $X$ and let $\mu \in \mathcal{M} (X)$ be an $F^{*}$ invariant measure. We say that $F$ is ergodic with respect to $\mu$, if for any $B \in \mathscr{B} (X)$ such that $B$ is almost invariant with respect to $F$ and $\mu$, we have $\mu (B) = 0$ or $\mu (B) = 1$.  
\end{definition}

\begin{theorem} [Londhe, \cite{londhe_ergodicity}]
\label{thm:dserg}
    Let $F$ be a holomorphic correspondence on $X$ satisfying the hypothesis in Theorem \ref{thm:ds} and let $\mu_F$ be the Dinh - Sibony measure associated with $F$. Then, $F$ is ergodic with respect to $\mu_F$.
\end{theorem}

Londhe also proved an analogue of the Birkhoff ergodic theorem for holomorphic correspondences, which is the main result of \cite{londhe_ergodicity}. 

\begin{theorem} [Londhe, \cite{londhe_ergodicity}]
\label{thm:ergthm}
    Let $F$ be a holomorphic correspondence on $X$ with topological degree $d$ and let $\mu \in \mathcal{M} (X)$. Let $F$ be ergodic with respect to $\mu$. Then for $\phi \in L^1(\mu)$, the sequence $\dfrac{1}{n} \displaystyle \sum_{j=0}^{n-1} \ \displaystyle \sideset{}{'}\sum_{w \, \in \, (F^{j})^{\dagger} (z)} \displaystyle \frac{\phi(w)}{d^j}$ converges $\mu$ almost everywhere to $\displaystyle{\int \limits_{X} \, \phi \, d\mu}$.
\end{theorem}

\subsection{The Koopman operator}
\label{Koopman op}

We end this section by introducing an operator that is important for our study. This operator was briefly mentioned in the work \cite{londhe_ergodicity}. This is the analogue of what is called the Koopman operator associated to a continuous map on a compact space. Thus, we call it as the Koopman operator associated to a holomorphic correspondence.

\begin{definition}
    Let $F$ be a holomorphic correspondence on $X$ with topological degree $d$ and let $\mu \in \mathcal{M} (X)$. For $1 \le q < \infty$, define $U_F: L^{q} (\mu) \to L^{q} (\mu)$ by $\left(U_F (\phi)\right) (z) = \dfrac{1}{d} \displaystyle \sideset{}{'} \sum \limits_{w \, \in \, F^{\dagger} (z)} \phi (w)$. The operator $U_F$ is called the Koopman operator associated to $F$. 
\end{definition}

It is easy to see that $U_F$ is linear. Further, one can consider the $n^{th}$ iterate of the operator $U_F$ which we denote by $U_F^{n}$. Observe that, $U_F ^{n} :  L^{q} (\mu) \to L^{q} (\mu)$ is given by $\left (U_{F}^{n} (\phi) \right) (z) = \dfrac{1}{d^n} \displaystyle \sideset{}{'} \sum \limits_{w \, \in \, \left(F^{n} \right)^{\dagger} (z)} \phi (w) = \left( U_{F^n}  (\phi) \right) (z)$. We now prove a basic, yet useful result about $U_F$. 

\begin{proposition}
\label{prop:contra}
    Let $F$ be a holomorphic correspondence on $X$ with topological degree $d$ and $\mu$ be an $F^{*}$ invariant measure. Then, $U_F$ is a contraction for any $1 \le q < \infty$.
\end{proposition}

\begin{proof}
   For $q=1$, the proposition follows from Equation \ref{eqn:inteq}. Now let $q > 1$ and let $\phi \in L^q$. Consider $\left \| U_F (\phi) \right \|_{q}^{q} = \displaystyle \bigintsss \limits_{X} \left | \dfrac{1}{d} \displaystyle \sideset{}{'} \sum \limits_{w \, \in \, F^{\dagger} (z)} \phi(w) \right |^q \, \mathrm{d}\mu.$ \\
    Choosing $p$ such that $\frac{1}{p} + \frac{1}{q} = 1$, applying H\"{o}lder's inequality and by using Equation \ref{eqn:inteq}, we get
    \[
    \left \| U_F (\phi) \right \|_{q}^{q} \leq \dfrac{d^{\frac{q}{p}}}{d^q} \bigintsss \limits_{X} \displaystyle \sideset{}{'} \sum \limits_{w \, \in \, F^{\dagger} (z)} |(\phi(w))^{q}| \, \mathrm{d}\mu =\dfrac{d^{\frac{q}{p}+1}}{d^q} \int \limits_{X} | \phi|^{q} \, \mathrm{d}\mu = \left \| \phi \right \|_{q}^{q},
    \]
    which completes the proof.
\end{proof}

It also follows that the linear operator $U_{F}^{n}$ is a contraction for any $n > 1$ as well.

\section{Mixing definition and motivation} 
\label{sec:three}

In this section, we define mixing notions related to a holomorphic correspondence. We further elaborate on the motivation behind our choice of this particular definition. Before going into that, let us first start from the case of maps. Consider a compact metric space $K$ and a map $T: K \to K$ that is invariant with respect to a Borel probability measure $\nu$ on $K$. A typical motivation for the definition of mixing of $T$ with respect to $\nu$ is given by the concept of correlations from statistics. 

For real valued functions $f, g \in L^{2} (\nu)$, we define the sequence of correlations corresponding to $f$ and $g$ to be,  
\begin{equation}
\label{eqn:correl}
    C_n (f, g) = \int \limits_{K} (f \circ T^n) g  \, d\nu - \int \limits_{K} f \, d\nu \int \limits_{K} g \, d\nu.
\end{equation}

Let $A$ and $B$ be Borel subsets of $K$. Setting $f = \chi_{\raisebox{-.7ex}{$\scriptstyle A$}}$ and $g = \chi_{\raisebox{-.7ex}{$\scriptstyle B$}}$ in the Equation \ref{eqn:correl}, we get $C_n (A, B) := C_n (\chi_{\raisebox{-.7ex}{$\scriptstyle A$}}, \chi_{\raisebox{-.7ex}{$\scriptstyle B$}}) = \nu \left( T^{-n} (A) \cap B \right) - \nu(A) \nu(B). \hspace{2mm} T$ is said to be mixing with respect to $\nu$ if the sequence $C_n (A, B)$ converges to $0$. In fact, it can be shown that it is equivalent to showing that the quantity $C_n(f,g)$ given in Equation \ref{eqn:correl} converges to $0$ for all $f, g \in L^2(\nu)$.

Note that as $T$ is invariant under $\nu$, we have $\nu (T^{-n} (A)) = \nu (A)$, for any $n \in \mathbb{N}$ and for any Borel subset $A \subseteq K$. Thus, if T is mixing, it essentially means that any two events are asymptotically independent. Inspired by this, we now develop the definition of mixing for holomorphic correspondences defined on a compact connected complex manifold $X$. The operator $U_T : L^{2} (\nu) \rightarrow L^{2} (\nu)$ given by $U_{T} (f) = f \circ T$ is called the Koopman operator associated to the map $T$. As one may readily observe, the iterates of $U_T$ are found in the definition of $C_n (f,g)$ as in Equation \ref{eqn:correl}. This motivates us to use the Koopman operator associated to a holomorphic correspondence, in order to generalize the definition of mixing to the case of holomorphic correspondences.

Let $F$ be a holomorphic correspondence on $X$ with $\mu$ being an $F^{*}$ invariant measure. Let $\phi, \, \psi \in L^{1} (\mu)$. For $n \in \mathbb{N}$, we define $I_{n}^{F} (\phi, \psi) := \displaystyle \int \limits_{X} \left(U_{F}^{n} (\phi)\right) \hspace{-1mm} (z) \psi(z) \, \mathrm{d} \mu.$ Also, we denote the integral $\displaystyle \int \limits_{X} \phi \, \mathrm{d}\mu$ by the notation $I(\phi)$. 

\begin{definition}
\label{defn:mixing}
Let $F$ be a holomorphic correspondence on $X$ with topological degree $d$ and let $\mu$ be an $F^{*}$ invariant measure. Then, $F$ is said to be 
\begin{enumerate}
    \item mixing with respect to $\mu$ if $\lim \limits_{n \, \to \, \infty} I_{n}^{F} (\phi, \psi) = I(\phi)I(\psi)$, for all $\phi \in \mathcal{C} (X, \mathbb{R})$ and $\psi \in L^{1} (\mu)$. 
    \item weakly mixing with respect to $\mu$ if $\lim \limits_{n \, \to \, \infty} \dfrac{1}{n} \displaystyle \sum \limits_{j \, =\, 0}^{n-1} \left|I_{j}^{F} (\phi, \psi) - I(\phi)I(\psi) \right| = 0$, for all $\phi \in \mathcal{C} (X, \mathbb{R})$ and $\psi \in L^{1} (\mu)$. 
\end{enumerate}
\end{definition}

The following theorem identifies a nice class of correspondences that satisfy the definition of mixing with respect to an appropriate measure. Much like the case of rational maps where the maps had the property of mixing with respect to the measures constructed in the equidistribution results in \cite{lyubich_equi} and \cite{flm_equi}, we see that the analogous measure obtained from Theorem \ref{thm:ds} by Dinh and Sibony, also satisfies the definition of mixing. 

\begin{theorem}
\label{thm:dsmix}
Let $F$ be a holomorphic correspondence on $X$, satisfying the hypothesis of Theorem \ref{thm:ds}. Let $\mu_F$ denote the Dinh - Sibony measure associated with $F$. Then $F$ is mixing with respect to $\mu_F.$ 
\end{theorem}

\begin{proof}
Let $\phi \in \mathcal{C} (X, \mathbb{R})$ and $\psi \in L^{1} (\mu)$. By Theorem \ref{thm:ds}, 
\[
\lim \limits_{n \, \to \, \infty}\left( \frac{1}{d^{n}} \displaystyle \sideset{}{'}  \sum\limits_{w \, \in \, (F^n){\dagger}(z)} \phi(w)\right)= \int \limits_{X} \phi \, \mathrm{d}\mu_F, \ \text{for} \ \mu - \text{a.e} \ \ z \in X.
\]
Multiplying on both sides by $\psi(z)$ and then integrating (the integral is essentially over the set of full measure where above convergence happens) we get
\[
\bigintsss \limits_{X} \lim \limits_{n \, \to \, \infty} \left( \frac{1}{d^{n}} \displaystyle \sideset{}{'}  \sum\limits_{w \, \in \, (F^n){\dagger}(z)} \phi(w)\right)\psi(z) \, \mathrm{d}\mu_F= \int \limits_{X} \phi \, \mathrm{d}\mu_F \int \limits_{X} \psi \, \mathrm{d}\mu_F. 
\]
Now, an application of the dominated convergence theorem completes the proof.
\end{proof}

We now make a few comments regarding the motivation behind Definition \ref{defn:mixing}. A simple verification shows that, $I_n ( \chi_{\raisebox{-.7ex}{$\scriptstyle A$}},  \chi_{\raisebox{-.7ex}{$\scriptstyle B$}}) \le \displaystyle \int \limits_{X} \chi_{\raisebox{-.7ex}{$\scriptstyle F^n(A)$}} \chi_{\raisebox{-.7ex}{$\scriptstyle B$}} \, \mathrm{d}\mu = \mu(F^n(A) \cap B)$. In the case when $\mu = \mu_F$, the Dinh - Sibony measure associated with $F$ (which is one of the most well studied invariant measures with regards to a holomorphic correspondence), using Theorem \ref{thm:ds} and the dominated convergence theorem, one can show that $\lim \limits_{n \, \to \, \infty} I_n ( \chi_{\raisebox{-.7ex}{$\scriptstyle A$}},  \chi_{\raisebox{-.7ex}{$\scriptstyle B$}}) = \mu (A) \mu(B)$. This gives, $\lim \limits_{n \, \to \, \infty} \mu(F^n(A) \cap B) \ge \mu(A) \mu(B)$, provided the limit exists. Combining with what we obtained in Theorem \ref{thm:dsmix}, we observe that the convergence of the sequence $I_n ( \chi_{\raisebox{-.7ex}{$\scriptstyle A$}},  \chi_{\raisebox{-.7ex}{$\scriptstyle B$}})$ and that of the sequence $\mu(F^n(A) \cap B)$ to appropriate limits are not equivalent, unlike the case of maps. This marks a significant change in the phenomenon of mixing in the case of maps from that of correspondences.

Here are a few heuristic justifications to convince the reader of why this hurdle appears. It mainly arises because of the pull back construction of the notion of invariant measures, which in turn is necessitated due to the multi - valued nature of $F$. Note that, unlike in the case of maps, for a holomorphic correspondence $F$ with $\mu$ being any $F^{*}$ invariant measure, we only have $\mu(A) \le \mu(F^{\dagger} (A))$ and $\mu(A) \le \mu(F(A))$ for all $A \in \mathscr{B}(X)$ (the interested reader may refer to Lemma 5.5 and Lemma 5.6 in \cite{londhe_recurrence} for a proof of the same). Thus, the idea of measuring asymptotic independence using measures of the events $F^n(A)$ and $B$ breaks down in this setting. However, $F^{*}$ invariance property is known to work well with regards to the use of integrals, since the measure $F^{*}\mu$ itself is constructed via the Riesz representation theorem. This strategy can be found to be used in a few results in the works like \cite{londhe_ergodicity}, \cite{bs_p1} and \cite{dinhsibony_og}. Thus, using the idea of correlations determined using integrals, we work with the above definition (the one given in Definition \ref{defn:mixing}) of mixing throughout the paper. The spirit of these heuristic justifications are further captured in the proofs of Propositions \ref{thm:ineqeg} and \ref{prop:aiislim} and Example \ref{example}.

\section{Ergodicity through the lens of mixing}
\label{sec:four}
In this section we prove various results that show some interplay between the notions of ergodicity and that of mixing in holomorphic correspondences. We give characterisation results based on integral and measure theoretic techniques. As promised in Section \ref{sec:three}, we also state and prove a few results that resonate with the heuristic justifications provided there. We begin with a theorem that characterises ergodicity of holomorphic correspondences.
\begin{theorem}
\label{thm:ergequiv}
    Let $F$ be a holomorphic correspondence on $X$ with $\mu$ being an $F^{*}$ invariant measure. Let $d$ be the topological degree of $F$. Then the following are equivalent.
    \begin{enumerate}
        \item $F$ is ergodic with respect to $\mu$.
        \item For any $\phi,\psi \in L^2(\mu)$, we have $\lim \limits_{n \, \to \, \infty}\dfrac{1}{n} \sum \limits_{j=0}^{n-1} I_{j}^{F} (\phi, \psi) = I(\phi) I(\psi).$
\item For any $\phi\in \mathcal{C}(X,\mathbb{R})$ and $\psi \in L^1(\mu)$, we have
$\lim \limits_{n \, \to \, \infty}\dfrac{1}{n} \sum \limits_{j=0}^{n-1} I_{j}^{F} (\phi, \psi) = I(\phi) I(\psi).$
    \end{enumerate}
\end{theorem}

\begin{proof}
\textbf{1 $\implies$ 3} \ Let $\phi\in \mathcal{C}(X,\mathbb{R})$ and $\psi \in L^1(\mu)$. As $F$ is ergodic with respect to $\mu$, by Theorem \ref{thm:ergthm}, 
\begin{equation}
\label{eqn:ergcgt}
\lim \limits_{n \, \to \, \infty}\dfrac{1}{n} \displaystyle \sum_{j=0}^{n-1} \ \displaystyle \sideset{}{'}\sum_{w \, \in \, (F^{j})^{\dagger} (z)} \displaystyle \frac{\phi(w)}{d^j} = \int \limits_{X} \phi \, \mathrm{d}\mu, \ \text{for} \ \mu - \text{a.e} \ \ z \in X.
\end{equation}
Applying the proof strategy of Theorem \ref{thm:dsmix}, to the Equation \ref{eqn:ergcgt} completes the proof of this implication. 

\textbf{3 $\implies$ 2} \ Let $\phi,\psi \in L^2(\mu)$. Thus, clearly, $\psi \in L^1(\mu). $ By density of $\mathcal{C}(X,\mathbb{R})$ in $L^2$, given any $\epsilon > 0$, there exists $\eta \in \mathcal{C}(X,\mathbb{R})$ such that $\lVert \phi - \eta \rVert_2 < \epsilon.$ Now consider
\begin{eqnarray}
\label{eqn:csineq}
\bigg|\frac{1}{n} \sum \limits_{j=0}^{n-1} I_{j}^{F} (\phi, \psi) - I(\phi) I(\psi)\bigg| & = & \bigg|\frac{1}{n} \sum \limits_{j=0}^{n-1} I_{j}^{F} (\phi, \psi) - \frac{1}{n} \sum \limits_{j=0}^{n-1} I_{j}^{F} (\eta, \psi) + \frac{1}{n} \sum \limits_{j=0}^{n-1} I_{j}^{F} (\eta, \psi) \nonumber \\
& & \hspace{3cm} - I(\eta) I(\psi) + I(\eta) I(\psi) - I(\phi) I(\psi)\bigg| \nonumber \\
& \leq & \bigg|\frac{1}{n} \sum \limits_{j=0}^{n-1} I_{j}^{F} (\phi-\eta, \psi)\bigg|+\bigg|\frac{1}{n} \sum \limits_{j=0}^{n-1} I_{j}^{F} (\eta, \psi) - I(\eta) I(\psi) \bigg| \nonumber \\
& &  \hspace{5.8cm} + \bigg|I(\psi) I(\eta - \phi)\bigg| 
\end{eqnarray}
The terms in the Equation \ref{eqn:csineq} can be estimated by using the Cauchy Schwartz inequality, Proposition \ref{prop:contra}, the hypothesis of this part of implication and the estimate $\lVert \phi - \eta \rVert_2 < \epsilon.$ Thus, we choose $n_0 \in \mathbb{N}$ such that for all $n \ge n_0,$ $\bigg|\dfrac{1}{n} \displaystyle \sum \limits_{j=0}^{n-1} I_{j}^{F} (\phi, \psi) - I(\phi) I(\psi)\bigg| < 2\epsilon||\psi||_2 + \epsilon$, which completes the proof of this part of the theorem.

\textbf{2 $\implies$ 1} \ Let $B$ be an almost invariant set with respect to $F$ and $\mu.$ Choosing $\phi := \chi_{\raisebox{-.7ex}{$\scriptstyle B$}}$ and $\psi := \chi_{\raisebox{-.7ex}{$\scriptstyle B^c$}}$, we see that, 
\begin{equation}
\label{eqn:intchi}
\lim \limits_{n \, \to \, \infty} \dfrac{1}{n} \displaystyle \sum \limits_{j=0}^{n-1} I_{j}^{F} (\chi_{\raisebox{-.7ex}{$\scriptstyle B$}}, \chi_{\raisebox{-.7ex}{$\scriptstyle B^c$}}) = I(\chi_{\raisebox{-.7ex}{$\scriptstyle B$}}) I(\chi_{\raisebox{-.7ex}{$\scriptstyle B^c$}}).
\end{equation}
As $B$ and $B^{c}$ are almost invariant with respect to $F^j$ and $\mu$, $\dfrac{1}{d^{j}} \displaystyle \sideset{}{'}  \sum\limits_{w \, \in \, \left(F^j \right)^{\dagger}(z)} \chi_{\raisebox{-.7ex}{$\scriptstyle B$}}(w)=\chi_{\raisebox{-.7ex}{$\scriptstyle B$}}(z) \ \text{for} \ \mu - \text{a.e} \\ z \in X$, for every $j \in \mathbb{N}$. Applying this observation in Equation \ref{eqn:intchi}, we get 
\[
\lim \limits_{n \, \to \, \infty} \dfrac{1}{n} \displaystyle \sum \limits_{j \, = \, 0}^{n - 1} \displaystyle \int \limits_{X} \chi_{\raisebox{-.7ex}{$\scriptstyle B$}} \chi_{\raisebox{-.7ex}{$\scriptstyle B^c$}} \, \mathrm{d}\mu = \displaystyle \int \limits_{X} \chi_{\raisebox{-.7ex}{$\scriptstyle B$}} \, \mathrm{d}\mu \displaystyle \int \limits_{X} \chi_{\raisebox{-.7ex}{$\scriptstyle B^c$}} \, \mathrm{d}\mu.
\]
This results in $\mu (B) = 0$ or $\mu (B^c) = 0$, which proves that $F$ is ergodic with respect to $\mu$.
\end{proof}

Upon comparing the third statement in Theorem \ref{thm:ergequiv} and the Definition \ref{defn:mixing}, one can observe that ergodicity of correspondences can be interpreted as ``mixing on the average". Next we state a theorem, regarding similar equivalence conditions for when a correspondence is mixing and weakly mixing respectively.
\begin{theorem}
\label{thm:mixerqui}
    Let $F$ be a holomorphic correspondence on $X$ with $\mu$ being an $F^{*}$ invariant measure. Let $d$ be the topological degree of $F$. Then the following are equivalent.
    \begin{enumerate}
        \item $F$ is mixing (resp. weakly mixing) with respect to $\mu$.
        \item For any $\phi,\psi \in L^2(\mu)$,\\
$\lim \limits_{n \, \to \, \infty} I_{n}^{F} (\phi, \psi) = I(\phi) I(\psi)$, $\left(\text{resp}. \lim \limits_{n \, \to \, \infty}\dfrac{1}{n} \displaystyle \sum \limits_{j=0}^{n-1} \big|I_{j}^{F} (\phi, \psi) - I(\phi) I(\psi)\big| = 0 \right)$.
\end{enumerate}
\end{theorem}

\begin{proof}
The proof of both the implications is similar to the strategy used in proving parts of Theorem \ref{thm:ergequiv}. Thus, we just mention a few comments about the proof without writing it fully.

\textbf{1 $\implies$ 2} \ The proof of this part works \emph{mutatis mutandis} the proof of \textbf{3 $\implies$ 2} in Theorem \ref{thm:ergequiv}.

\textbf{2 $\implies$ 1} \ In this case, we use the fact that $\mathcal{C} (X, \mathbb{R}) \subset L^{2} (\mu)$ and that any $L^1$ function can be approximated by a function in $L^2$. These estimates and the core strategy from the proof of
\textbf{3 $\implies$ 2} of Theorem \ref{thm:ergequiv} are enough to complete the proof.
\end{proof}

We now interpret Theorem \ref{thm:ergequiv} using measures. As we had seen in Section \ref{sec:three}, while motivating our definition of mixing, the sequences $I_n (\phi, \psi)$ and $\mu(F^{n}(A) \cap B)$ need not behave well with respect to each other in our dynamical system even as they connect well with each other in the case of maps. One can observe a similar pattern here with regards to the following theorem when it is compared with Theorem \ref{thm:ergequiv}. To be more precise, for a continuous map $T$ defined on a compact metric space $K$, ergodicity of $T$ with respect to an invariant Borel probability measure $\nu$ is equivalent to the statement that for any Borel sets $A_1, A_2 \subset K$, we have 
\begin{align}
\label{eqn:silva5.1.5}
\lim \limits_{n \to \infty} \dfrac{1}{n} \displaystyle \sum \limits_{j=0}^{n-1} \nu ( T^{-j} (A_1) \cap A_2) = \nu (A_1) \nu (A_2).
\end{align}
 
One can refer to a standard text in ergodic theory, say \cite{walters_et} or \cite{silva_invitation} for a proof of the same. The limit (if it exists) of the analogue of the above sequence, in the present scenario is described in terms of an inequality. Although they don't have any direct link to the main theorems of this work, Propositions \ref{thm:ineqeg} and \ref{prop:ergequal} are stated and proved in this paper to serve as a motivation to compare the situation in case of map and that of holomorphic correspondences.

\begin{proposition}
\label{thm:ineqeg}
Let $F$ be a holomorphic correspondence on $X$ with topological degree $d$ and let $\mu$ be an $F^{*}$ invariant measure. Suppose $F$ is ergodic with respect to $\mu$. Let $A,B \in \mathscr{B} (X)$ be such that $\lim \limits_{n \to \infty} \dfrac{1}{n} \displaystyle \sum \limits_{j=0}^{n-1} \mu ( F^j(A) \cap B)$ exists, then $\lim \limits_{n \to \infty} \dfrac{1}{n} \displaystyle \sum \limits_{j=0}^{n-1} \mu ( F^j(A) \cap B) \geq \mu (A) \mu (B)$.
\end{proposition}

\begin{proof}
Suppose $F$ is ergodic with respect to $\mu$ and let $A, B \in \mathscr{B} (X)$. As $\chi_{\raisebox{-.7ex}{$\scriptstyle A$}}$ is integrable, by Theorem \ref{thm:ergthm} we have $\lim \limits_{n \to \infty} \dfrac{1}{n} \displaystyle \sum \limits_{j=0}^ {n-1} \displaystyle \sum \limits_{ w \, \in \, (F^j)^{\dagger} (z) } \dfrac{\chi_{\raisebox{-.7ex}{$\scriptstyle A$}} (w)}{d^j} = \mu(A)$ for $\mu - $ a.e  $z\in X$.
Multiplying both sides of the above equation by $\chi_{\raisebox{-.7ex}{$\scriptstyle B$}} (z)$ results in $\lim \limits_{n \to \infty} \dfrac{1}{n} \displaystyle \sum_{j=0}^ {n-1} \displaystyle \sum_{ w \in (F^j)^\dagger(z) } \dfrac{\chi_{\raisebox{-.7ex}{$\scriptstyle A$}} (w)}{d^j} \chi_{\raisebox{-.7ex}{$\scriptstyle B$}} (z) = \mu(A) \chi_{\raisebox{-.7ex}{$\scriptstyle B$}} (z), \ \text{for} \ \mu -  \text{a.e} \ z \in X.$

Now, for any $n>0$, $ \displaystyle \dfrac{1}{n}\sum_{j=0}^ {n-1} \sum_{ w \in (F^j)^\dagger(z) } \dfrac{\chi_{\raisebox{-.7ex}{$\scriptstyle A$}}(w)}{d^j} \chi_{\raisebox{-.7ex}{$\scriptstyle B$}} (z)$  is bounded by 1. Making use of the dominated convergence theorem, we get

\begin{eqnarray*}
    \lim_{n \to \infty} \bigintsss \limits_{X} \frac{1}{n}\sum_{j=0}^ {n-1} \sum_{ w \in (F^j)^\dagger(z) } \frac{\chi_{\raisebox{-.7ex}{$\scriptstyle A$}} (w)}{d^j} \chi_{\raisebox{-.7ex}{$\scriptstyle B$}}(z) \,\mathrm{d}\mu(z) & = & \bigintsss \limits_{X} \lim_{n \to \infty}  \frac{1}{n}\sum_{j=0}^ {n-1} \sum_{ w \in (F^j)^\dagger(z) } \frac{\chi_{\raisebox{-.7ex}{$\scriptstyle A$}}(w)}{d^j} \chi_{\raisebox{-.7ex}{$\scriptstyle B$}} (z) \, \mathrm{d}\mu(z)\\
& = & \mu (A)\mu (B).
\end{eqnarray*}
Now,
\begin{eqnarray*}
    \bigintsss \limits_{X} \frac{1}{n}\sum_{j=0}^ {n-1} \sum_{ w \in (F^j)^\dagger(z)} \frac{\chi_{\raisebox{-.7ex}{$\scriptstyle A$}} (w)}{d^j} \chi_{\raisebox{-.7ex}{$\scriptstyle B$}} (z) \, \mathrm{d}\mu(z)
     & = & \frac{1}{n} \sum_{j=0}^{n-1} \bigintsss \limits_{X} \chi_{\raisebox{-.7ex}{$\scriptstyle B$}} (z) \sum_{ w \in (F^j)^\dagger(z)} \frac{\chi_{\raisebox{-.7ex}{$\scriptstyle A$}}(w)}{d^j} \ \mathrm{d}\mu(z)\\
     &\leq & \frac{1}{n} \sum_{j=0}^{n-1} \int \limits_{X} \chi_{\raisebox{-.7ex}{$\scriptstyle B\cap (F^j) (A)$}} (z) \, \mathrm{d}\mu(z)\\
     & = &\frac{1}{n} \sum_{j=0}^{n-1} \mu(B\cap F^j (A)).
\end{eqnarray*}
Hence, the result holds. 
\end{proof}

Note that the Proposition \ref{thm:ineqeg} uses Theorem \ref{thm:ergthm} in its proof. So, the analogue of the sequence involved in the Equation \ref{eqn:silva5.1.5} involves forward iterates of a holomorphic correspondence $F$ unlike the backward iterates as in the case of maps. The reason for the same has to do with the pull back invariance being considered in the case of correspondences even as we work with push forward invariance in case of maps. This brings in a change in the sequence involved in the ergodic theorem, which is reflected in the following proposition as well. However, one may consider $F_{*}$ invariance of an appropriate $\mu$, as in use the concept of push forward of $\mu$ under $F$ to study the sequence in terms of backward images. For more details on this, one may consult Remark 1.3 in \cite{londhe_ergodicity}. 

We now mention an example in which we consider what is known as a rational semigroup. This concept was introduced and first extensively studied in the work \cite{hink_martin}, where the reader can find the definition of a rational semigroup, the Julia set associated to it and its properties. This example uses ideas mentioned in \cite{hink_martin} and \cite{Boyd}. The interested reader may refer to \cite{bs_semigrp} and \cite{bs_entropy} for more details on the connection between rational semigroups and holomorphic correspondences.

\begin{example}
\label{example}
Consider the finitely generated rational semigroup $S = \left \langle z^2, \frac{z^2}{2} \right \rangle$. Then, there exists $A, B \in \mathscr{B} (\widehat{\mathbb{C}})$ such that $\lim \limits_{n \to \infty} \dfrac{1}{n} \displaystyle \sum \limits_{j=0}^{n-1} \mu ( F^j(A) \cap B) > \mu (A) \mu (B).$
\end{example}

\begin{proof}
First, note that one can view $S$ as a holomorphic correspondence on $\widehat{\mathbb{C}}$. To do this, consider $P_1 (z,w) = w - z^2$ and $P_2 (z,w)= w - \frac{z^2}{2}$. Let $\Gamma_1$ and $\Gamma_2$ be the subvarieties of $\widehat{\mathbb{C}} \times \widehat{\mathbb{C}}$ defined by the polynomials $P_1$ and $P_2$ respectively. Then, $\Gamma = \Gamma_1 + \Gamma_2$ is the correspondence associated to $S$. 

The Julia set of $S$ to be denoted by $J(S)$ is given by the set $\{z\in\mathbb{C}: 1 \leq |z| \leq 2 \}$. Let $\mu$ denote the Dinh Sibony measure associated with the correspondence $\Gamma$. It can be observed that the support of $\mu$ is the set $J(S)$. From \cite{Boyd}, we know that, $\mu(B) = \dfrac{m(log(B))}{2\pi \, log 2}, \ \text{for any Borel set} \ B \subset J(S)$. Here, $m$ is the Lebesgue measure on $\mathbb{C}$. 

Suppose $A = \{z\in\mathbb{C}: |z| > \sqrt{2} \}$, then we have $A \subset F^j(A)$, for every $j\geq 0.$ Thus we get $\lim \limits_{n \to \infty} \dfrac{1}{n} \displaystyle \sum \limits_{j=0}^{n-1} \mu ( F^j(A) \cap A) = \mu(A)$. A simple calculation shows that $\mu (A) = \frac{1}{2}$ and hence $\mu(A)^{2} < \mu(A)$. To be more specific, for the above choice of $A$, we have $\lim \limits_{n \to \infty} \dfrac{1}{n} \displaystyle \sum \limits_{j=0}^{n-1} \mu ( F^j(A) \cap A) > \mu (A) \mu (A).$ 
\end{proof}

This example essentially shows that the inequality obtained in the conclusion of Proposition \ref{thm:ineqeg} can be strict - thereby resulting in a different scenario than the one observed in the case of maps. The above example also indicates that our definition of mixing of holomorphic correspondences (considered with $X$ being equal to $\widehat{\mathbb{C}}$) works for the case of rational semigroups as well. Essentially, the dynamics of holomorphic correspondences helps in studying some dynamical properties of rational semigroups. It must be noted that, a certain skew product map is another typical way to model the dynamics of a rational semigroup. In \cite{sumi_skew}, Sumi investigated this in a greater detail and also speaks of exactness of the skew product map with respect to an invariant measure that he constructs in that article. In the case of a measure preserving transformation, exactness is a stronger notion than that of mixing. Thus, Sumi's result implied that the skew product map corresponding to a rational semigroup is mixing with respect to an appropriate measure. 

Now, we proceed to proving that the limit we consider in Proposition \ref{thm:ineqeg} exists when the set $B$ is almost invariant with respect to $F$ and $\mu$. This gives a good number of cases where Proposition \ref{thm:ineqeg} can be applied.  
 
\begin{proposition}
\label{prop:aiislim}
    Let $F$ be a holomorphic correspondence on $X$ with topological degree $d$ and let $\mu$ be an $F^{*}$ invariant measure. Let $A,B \in \mathscr{B} (X)$, where $B$ is almost invariant with respect to $F$ and $\mu.$ Then,  $\lim \limits_{n \to \infty} \dfrac{1}{n} \displaystyle \sum \limits_{j=0}^{n-1} \mu ( F^j(A) \cap B)$ exists.
\end{proposition}

\begin{proof}
It is enough to prove that the sequence $(a_j)=\mu ( F^j(A) \cap B)$ converges. Since $(a_j)$ is bounded above by $\mu(B)$, showing it is an increasing sequence completes the proof.

Consider the measure $\mu':=\frac{1}{\mu(B)}\mu|_B$. Whenever $\mu$ is an $F^{*}$ invariant measure and $B$ is an almost invariant set with respect to $F$ and $\mu$, we have $\mu'$ is also $F^{*}$ invariant measure (one may refer the proof of Corollary 5.4 in \cite{londhe_ergodicity} for the same). Thus we have $\mu'(F^j(B)) \leq \mu'(F^{j+1}(B))$, for any $j \in \mathbb{N} \cup \{0\}$, thanks to Lemma 5.6 in \cite{londhe_recurrence}. Now by rewriting the sequence $(a_j)$ as $\mu(B).\mu'(F^j(A))$, it's easy to observe that $(a_j)$ is an increasing sequence. Hence, $\lim \limits_{n \to \infty} \dfrac{1}{n} \displaystyle \sum \limits_{j=0}^{n-1} \mu ( F^j(A) \cap B)$ exists.  
\end{proof}

Note that almost invariance of $B$ is just a sufficient condition for the existence of the limit. The Borel sets considered in Example \ref{example} shows that the limit may exists otherwise too.

\begin{proposition}
\label{prop:ergequal}
Let $F$ be a holomorphic correspondence on $X$ with topological degree $d$ and $\mu$  be an $F^{*}$ invariant Borel probability measure. If
\begin{equation}
\label{eqn:10}
\lim_{n \to \infty} \frac{1}{n}\sum_{j=0}^{n-1} \mu ( F^j(A) \cap C) = \mu (A) \mu (C)
\end{equation}
for all measurable sets $A,C$, then $F$ is ergodic with respect to $\mu$.
\end{proposition}
\begin{proof}
Let $B$ be an almost invariant set with respect to $F$ and $\mu$. Then we have $B$ is almost invariant with respect to $F^i \ \text{and} \ \mu,   \ \text{for all} \ i \in \mathbb{N}$. Therefore, by Londhe's construction (see Lemma 3.1 in \cite{londhe_ergodicity}), we have a decreasing sequence of Borel sets, $B_i' \subset B$ such that $(F^i)^{\dagger}(B_{i}')  \subset  B$ and $\mu(B_i')=\mu(B)$ for all $i \in \mathbb{N}$. Consider the set, $ B_\infty := \bigcap \limits_{i=0}^{\infty} B_i^{'}$, where we set $B_0^{'}= B$. Now, fix $n \in \mathbb{N}$. For $0 \leq j \leq n-1$ we have, $(F^j)^{\dagger}(B_\infty)  \subset  (F^j)^{\dagger}(B_{j}')  \subset  B$ and hence $B_\infty \subseteq F^j(B)$. Now, substituting  $A = B$ and $C = B_\infty$ in Equation \ref{eqn:10}  gives $\mu(B) = \mu(B)^2$ which implies $\mu(B) = 0 \text{ or } \mu(B)= 1$. Hence, $F$ is ergodic with respect to $\mu$.
\end{proof}

We conclude this section by looking at the connection between the two different types of mixing and ergodicity. The following Lemma is a basic result (one may refer Lemma 6.1.1 in \cite{silva_invitation}) that will straightaway exhibit the connection that we are looking for.

\begin{lemma}
\label{lem:cesaro}
    Let $\{a_n\}$ be a bounded sequence. Then 
    \begin{enumerate}
        \item 
        $\lim \limits_{n \, \to \, \infty} a_n = a \implies \lim \limits_{n \, \to \, \infty} \dfrac{1}{n} \displaystyle \sum \limits_{j=0}^{n-1} |a_j - a| = 0.$
        \item
        $\lim \limits_{n \, \to \, \infty} \dfrac{1}{n} \displaystyle \sum \limits_{j=0}^{n-1} |a_j - a| = 0 \implies \lim \limits_{n \, \to \, \infty} \dfrac{1}{n} \displaystyle \sum \limits_{j=0}^{n-1} a_j = a.$
    \end{enumerate} 
\end{lemma}

\begin{theorem}
\label{thm:equiv}
    Let $F$ be a holomorphic correspondence on $X$ with topological degree $d$, and let $\mu \in \mathcal{M} (X)$ be an $F^{*}$ invariant measure. Then, with respect to the measure $\mu$,
\begin{center}
    $F \ \text{is mixing}   \implies F \ \text{is weakly mixing} \implies F \ \text{is ergodic}$ 
\end{center}
\end{theorem}

\begin{proof}
Let $F$ be mixing with respect to $\mu$. For $\phi \in \mathcal{C} (X, \mathbb{R})$ and $\psi \in L^{1} (\mu)$, consider the sequence $a_n = I_{n}^{F} (\phi, \psi)$. By definition, $a_n$ converges to $I(\phi) I(\psi)$. Now, by the first statement of Lemma \ref{lem:cesaro}, $F$ is weakly mixing with respect to $\mu$. Similarly, the second statement in Lemma \ref{lem:cesaro} and Theorem \ref{thm:ergequiv} proves that if $F$ is weakly mixing with respect to $\mu$ then it is ergodic with respect to $\mu$.   
\end{proof}

\section{Product of holomorphic correspondences}
\label{sec:five}

In this section, we introduce the concept of product of two holomorphic correspondences. Before that we take a brief detour to the case of continuous maps in order to understand the motivation for studying product of two holomorphic correspondences in this work. Suppose $T_1: K_1 \to K_1$ and $T_2: K_2 \to K_2$ are continuous functions, where the spaces $K_1$ and $K_2$ are compact. Then, the product $T_1 \times T_2$ is defined as $T_1 \times T_2: K_1 \times K_2 \to K_1 \times K_2$ given by $(T_1 \times T_2) (x_1, x_2) = (T_1(x_1), T_2(x_2))$. The details regarding the dynamics of this map in relation to that of $T_1$ and $T_2$ can be found in a standard ergodic theory book, say for example, \cite{walters_et}. Our main goal in this section is to develop an analogous theory in the study of holomorphic correspondences. We state the definition of product of two holomorphic correspondences - which shall readily indicate that the resulting expression of the product is again a holomorphic correspondence. We also analyse its dynamics and measure theoretic aspects. Product of holomorphic correspondences has an interesting application in characterising weakly mixing of holomorphic correspondences - as we shall see in the next section. It must be noted that, the product of holomorphic correspondences were briefly considered in the recent preprint \cite{luoMarco2026} as well. We were able to improve our presentation of product of two holomorphic correspondences in this paper after going through the above preprint and having a correspondence with its second named author.

Let $X_1$ and $X_2$ be compact connected complex manifolds. Let $F_1$ and $F_2$ be holomorphic correspondences on $X_1$ and $X_2$ respectively. Let us denote their respective subvariety representations by $\Gamma^1=\displaystyle \sum\limits_{1\leq i\leq N_1} m_i\Gamma_i^1$ and  $\Gamma^2=\displaystyle \sum\limits_{1\leq j\leq N_2} n_j\Gamma_j^2$. Consider the map $S: X_1 \times X_2 \times X_1 \times X_2 \longrightarrow X_1 \times X_1 \times X_2 \times X_2$ given by $S(z_1, z_2, z_3, z_4) = (z_1, z_3, z_2, z_4)$. Clearly, $S$ is a holomorphic function. Fix $1 \le i \le N_1$ and $1 \le j \le N_2$. From \cite{chirka}, we know that $\Gamma_i^{1} \times \Gamma_j^{2}$ is an irreducible complex subvariety of $X_1 \times X_1 \times X_2 \times X_2$ and consequently, $S^{-1} (\Gamma_i^{1} \times \Gamma_j^{2})$ is an irreducible complex subvariety of $X_1 \times X_2 \times X_1 \times X_2$. Now, define
\begin{equation}
\label{eqn:proddefn}
     \Gamma^{1} \times \Gamma^{2} := \sum \limits_{1\leq i\leq N_1} \, \sum \limits_{1\leq j\leq N_2} m_i n_j S^{-1} (\Gamma_i^{1} \times \Gamma_j^{2}).
\end{equation}
We call the expression in Equation \ref{eqn:proddefn} as the product of the holomorphic correspondences $\Gamma^{1}$ and $\Gamma^{2}$. From the discussion in the above paragraph, $\Gamma_1 \times \Gamma_2$ is a holomorphic correspondence on the compact connected complex manifold $X_1 \times X_2$, as it satisfies Definition \ref{defn:holcorresp}. Denoting the set valued map corresponding to $\Gamma_1 \times \Gamma_2$ by $F_1 \times F_2$, we see that, $(F_1 \times F_2) (z_1, z_2) = F_1 (z_1) \times F_2 (z_2)$. Now, let $d_1$ and $d_2$ be the topological degree of $F_1$ and $F_2$ respectively. Since, $(F_1\times F_2)^\dagger(w_1,w_2) = \left\{(z_1, z_2): z_1 \in F_1^\dagger(w_1) \ \text{and} \ z_2 \in F_2^\dagger(w_2) \right\},$ we see that $\#(F_1\times F_2)^\dagger(w_1,w_2) = (\# F_1^\dagger(w_1))(\# F_2^\dagger(w_2))$, for any generic point $(w_1,w_2)\in X_1 \times X_2.$ Therefore, topological degree of $F_1\times F_2 = d_1d_2$. We now state and prove a simple Proposition regarding the invariance of measures with respect to the product correspondence. 

\begin{proposition}
Let $F_1$ and $F_2$ be two holomorphic correspondences on compact connected complex manifolds $X_1$ and $X_2$ respectively and let $d_1\ \text{and} \ d_2$ be their respective topological degrees. Also, let $\mu_1 \in \mathcal{M}(X_1)$ and $\mu_2 \in \mathcal{M}(X_2)$ are such that $\mu_1$ is $F_{1}^{*}$ invariant and $\mu_2$ is $F_2^{*}$ invariant. Denote by $\mu$ the product measure $\mu_1 \times \mu_2$. Then we have
$((F_1\times F_2)^{n})^{*}\mu = d_1^n d_2^n \mu$, for any $n \in \mathbb{N}$. In particular, the measure $\mu$ is $(F_1 \times F_2)^{*}$ invariant.
\end{proposition}
\begin{proof}
Firstly, recall from \cite{londhe_ergodicity} that, if $F$ is a holomorphic correspondence on $X$ and $\nu$ is a $F^{*}$ invariant measure, then $\nu$ is $(F^n)^{*}$ invariant as well. Hence, with regards to our notations, $\mu_1$ and $\mu_2$ are $(F_1^{n})^{*}$ invariant and $(F_2^{n})^{*}$ invariant respectively.

We shall prove that, for $\phi \in \mathcal{C} (X_1 \times X_2, \mathbb{R}), \ \left \langle ((F_1\times F_2)^{n})^{*}\mu, \phi  \right \rangle = d_1^n d_2^n  \left \langle\mu, \phi \right \rangle.$ Note that,
\begin{equation*}
\left \langle ((F_1\times F_2)^{n})^{*}\mu, \phi  \right \rangle = \bigintsss \limits_{X_1\times X_2} \left( \displaystyle \sideset{}{'}  \sum\limits_{(w_1,w_2) \, \in \, ((F_1\times F_2)^{n})^{\dagger}(z_1,z_2)} \phi(w_1,w_2)\right) \, \mathrm{d}(\mu)
\end{equation*}
By Fubini's theorem we have,
\begin{eqnarray*}
    \left \langle ((F_1\times F_2)^{n})^*\mu, \phi  \right \rangle & = & \bigintsss \limits_{X_2} \bigintsss \limits_{X_1} \left( \displaystyle \sideset{}{'}  \sum\limits_{(w_1,w_2) \, \in \, ((F_1\times F_2)^{n})^{\dagger}(z_1,z_2)} \phi(w_1,w_2)\right) \, \mathrm{d}\mu_1(z_1) \mathrm{d}\mu_2(z_2)\\
    & = & \bigintsss \limits_{X_2} \displaystyle \sideset{}{'}  \sum\limits_{w_2 \, \in \, (F_2^{n})^{\dagger}(z_2)} \left( \bigintsss \limits_{X_1} \left( \displaystyle \sideset{}{'}  \sum\limits_{w_1 \, \in \, (F_1^{n})^{\dagger}(z_1)} \phi(w_1,w_2)\right) \, \mathrm{d}\mu_1(z_1)\right) \mathrm{d}\mu_2(z_2)\\
    & = & \bigintsss \limits_{X_2} \displaystyle \sideset{}{'}  \sum\limits_{w_2 \, \in \, (F_2^{n})^{\dagger}(z_2)} \left(d_1^n \bigintsss \limits_{X_1}  \phi(z_1,w_2) \, \mathrm{d}\mu_1(z_1)\right) \mathrm{d}\mu_2(z_2)\\
    & = & d_1^n \bigintsss \limits_{X_1} \left( \bigintsss \limits_{X_2}  \left( \displaystyle \sideset{}{'}  \sum\limits_{w_2 \, \in \, (F_2^{n})^{\dagger}(z_2)} \phi(z_1,w_2)\right) \, \mathrm{d}\mu_2(z_2) \right) \mathrm{d}\mu_1(z_1)\\
    & = & d_1^n d_2^n \int \limits_{X_1}  \int \limits_{X_2}   \phi(z_1,z_2) \, \mathrm{d}\mu_2(z_2) \, \mathrm{d}\mu_1(z_1) =  d_1^n d_2^n \int \limits_{X_1\times X_2} \phi \, \mathrm{d}\mu,
\end{eqnarray*}
which proves the result.
\end{proof}

For a holomorphic correspondence $F$ on $X$ one can consider the product $F \times F$. It is this product correspondence that we shall mostly consider in the forthcoming parts of our work. For the sake of brevity, henceforth, we will denote the product correspondence $F \times F$ on $X \times X$ by the symbol $F^{\times}$.

\section{An application of Product correspondences}
\label{sec:six}
In this section, we give an application of the product of holomorphic correspondences as we prove further equivalent notions of the concept of weakly mixing. An important question in this context is to see when ergodicity and weakly mixing are preserved by the product correspondence. It turns out that, in case of a map $T: K \to K$ (where $K$ is a compact metric space) $T$ is weakly mixing if and only if the product map $T \times T$ is ergodic. The main theorem of this section, namely Theorem \ref{thm:main} address the analogous question in the case of holomorphic correspondences. Before we state and prove it, we develop a few tools required for the same. We first state an useful result without proof. The proof of the same can be found in [\cite{walters_et}, Theorem 1.20]. 
\begin{proposition}
\label{prop:wal20}
Let $(a_n)_{n \in \mathbb{N}}$ be a bounded sequence of real numbers. Then the following are equivalent:
\begin{enumerate}
\item $\lim \limits_{n \, \to \, \infty} \dfrac{1}{n} \displaystyle \sum \limits_{j \, = \, 0}^{n-1} |a_{j}| = 0$.
\item There exists a density zero set $D \subset \mathbb{Z}^{+} \left( \text{that is to say} \
\dfrac{\# (D \cap \{0, 1, \cdots n-1\})}{n} \rightarrow 0 \right)$ such that $\lim \limits_{D \, \notni \, n \, \to \, \infty} a_n = 0$.  
\item $\lim \limits_{n \, \to \, \infty} \dfrac{1}{n} \displaystyle \sum \limits_{j \, = \, 0}^{n-1} |a_{j}|^{2} = 0$.
\end{enumerate}
\end{proposition}

As a consequence of this proposition, we obtain the following lemma that we make use of while proving Theorem \ref{thm:main}. 
\begin{lemma}
\label{lem:equi}
Let $F$ be a holomorphic correspondence on $X$ with topological degree $d$, and let $\mu \in \mathcal{M} (X)$ be an $F^{*}$ invariant measure. Then the following are equivalent:
\begin{enumerate}
\item $F$ is weakly mixing with respect to $\mu$.
\item For every $\phi \in \mathcal{C} (X, \mathbb{R})$ and $\psi \in L^{1} (\mu)$, there exists a density zero set $D(\phi, \psi)$ such that
\[
\lim \limits_{D(\phi, \psi) \, \notni \, n \, \to \, \infty} I_{n}^{F} (\phi, \psi) = I(\phi) I(\psi).
\]
\item For every $\phi \in \mathcal{C} (X, \mathbb{R})$ and $\psi \in L^{1} (\mu)$, 
\[
\lim \limits_{n \, \to \, \infty} \dfrac{1}{n} \displaystyle \sum \limits_{j \, = \, 0}^{n-1} \left| I_{n}^{F} (\phi, \psi) - I(\phi) I(\psi) \right|^{2} = 0.
\]
\end{enumerate}
\end{lemma}

\begin{proof}
Let $\phi \in \mathcal{C} (X, \mathbb{R})$ and $\psi \in L^{1} (\mu)$. Choose $a_n = I_{n}^{F} (\phi, \psi) - I(\phi) I(\psi).$ It is easy to see that $a_n$ is a bounded sequence. Now, applying Proposition \ref{prop:wal20} to the above choice of $a_n$ completes the proof of the lemma.
\end{proof}

This proposition conveys that the notion of weakly mixing behaves like the concept of mixing albeit in a density zero subset of set of non negative integers. However, the density zero set does depend on the functions $\phi$ and $\psi$. So, it is not quite there yet and hence still remains a weaker notion than that of mixing. 

We now state and prove a result that gives a convenient way to check if product of two holomorphic correspondences is weakly mixing. In fact, it ensures that it is enough to check the behaviour of the correlations on a certain ``generating" subfamilies of the concerned spaces rather than checking it on the whole space itself. 

To start with, we define the following subfamilies. In what follows, we will adopt the following notation. Consider two sets $C_1$ and $C_2$.  Let $f_1: C_1 \to \mathbb{R}$ and $f_2: C_2 \to \mathbb{R}$. We define $f_1 * f_2: C_1 \times C_2 \to \mathbb{R}$ as $(f_1 * f_2) (x_1, x_2) = f_1(x_1)f_2(x_2)$. 

Let $\mu \in \mathcal{M}(X)$. We define the following sets.
\begin{eqnarray*}
P_{1} & = & \left\{ \phi \in \mathcal{C} (X \times X, \mathbb{R}) : \phi = \phi_1 * \phi_2, \ \text{where} \ \phi_1, \phi_2 \in \mathcal{C} (X, \mathbb{R}) \right\} \\
P_{2} & = & \left \{ \psi \in L^{1} (\mu \times \mu) : \psi = \chi_{\raisebox{-.7ex}{$\scriptstyle B_1$}} * \chi_{\raisebox{-.7ex}{$\scriptstyle B_2$}} = \chi_{\raisebox{-.7ex}{$\scriptstyle B_1 \times B_2$}}, \ \text{where} \ B_1, B_2 \in \mathscr{B} (X) \right\} \\
S_{1} & = & \left \{ \phi \in \mathcal{C} (X \times X, \mathbb{R}) : \phi = \sum \limits_{j \, = \, 1}^{k_1} \phi_{j}, \ \text{for some} \ k_1 \in \mathbb{Z}^{+} \ \text{and} \ \phi_{j} \in P_1, \ \forall \ 1 \, \le \, j \, \le \, k_1 \right \} \\
S_{2} & = & \left \{ \psi \in L^{1} (\mu \times \mu) : \psi = \sum \limits_{j \, = \, 1}^{k_2} a_j \psi_{j}, \ \text{for some} \ k_2 \in \mathbb{Z}^{+} \ \text{and} \ \psi_{j} \in P_2, \ \forall \ 1 \, \le \, j \, \le \, k_2  \right \} 
\end{eqnarray*}

\begin{remark}
\label{rem:approxcross}
An application of the Stone Weierstrass theorem implies that the set $S_1$ is dense in $\mathcal{C}(X \times X, \mathbb{R)}$. Further, any function in $L^{1} (\mu \times \mu)$ can be approximated by a function from the set $S_2$. We will apply these two facts in proving the following theorem.
\end{remark}

\begin{lemma}
\label{thm:alg}
Let $F$ be a holomorphic correspondence on $X$ with topological degree $d$ and $\mu \in \mathcal{M} (X)$ be an $F^{*}$ invariant measure. Then, $F^{\times}$ is weakly mixing with respect to $\mu \times \mu \, \iff \,$ for any $\phi \in P_1$ and $\psi \in P_2$, we have 
\begin{equation}
\label{eqn:mixcross}
\lim \limits_{n \, \to \, \infty} \dfrac{1}{n} \displaystyle \sum \limits_{j \, = 0}^{n - 1} \left | I_{j}^{F^{\times}} (\phi, \psi) - I (\phi) I (\psi) \right| = 0.   
\end{equation}  
\end{lemma}

\begin{proof}
The forward implication simply follows from the definition of weakly mixing. 

Before starting with the proof of the backward implication, observe that if any $\phi \in P_{1} \ \text{and} \ \psi \in P_2$ satisfies Equation \ref{eqn:mixcross}, then Equation \ref{eqn:mixcross} holds for all $\phi \in S_1$ and $\psi \in S_2$ as well. 

Now, to prove the backward implication, let us consider $\phi \in \mathcal{C} (X \times X, \mathbb{R})$ and $\psi \in L^{1} (\mu \times \mu)$. Let $\epsilon > 0$. By Remark \ref{rem:approxcross}, there exists $\phi_{1} \in S_1$ such that $\lVert \phi - \phi_1 \rVert_{\infty} < \epsilon$ and there exists $\psi_1 \in S_2$ such that, $\displaystyle \int \limits_{X \times X} |\psi - \psi_1| < \epsilon$. Note that, by choice of $\phi_1$ and $\psi_1$,
\begin{eqnarray}
\label{eqn:6.3}
\left | I (\phi_1) I (\psi_1) - I (\phi) I (\psi_1) \right| & = & \left | I (\psi_1) \right| \left | I (\phi_1) - I (\phi) \right| \nonumber \\ 
& \le & M_{1} \int \limits_{X \times X} \left |\phi_1 - \phi \right| \, \mathrm{d}(\mu \times \mu) \le M_1 \epsilon, \\
\label{eqn:6.4}
\left | I (\phi) I (\psi_1) - I (\phi) I (\psi) \right|  & = & \left | I (\phi) \right| \left | I (\psi_1) - I (\psi) \right| \nonumber \\
& \le & M_{2} \int \limits_{X \times X} \left |\psi_1 - \psi \right| \, \mathrm{d}(\mu \times \mu) \le M_2 \epsilon.
\end{eqnarray}

Further, we also have
\begin{eqnarray}
\label{eqn:6.5}
\left | I_{j}^{F^{\times}} (\phi, \psi) - I_{j}^{F^{\times}} (\phi_1, \psi_1) \right| & \le &  \left | I_{j}^{F^{\times}} (\phi, \psi) - I_{j}^{F^{\times}} (\phi, \psi_1) \right| + \left | I_{j}^{F^{\times}} (\phi, \psi_1) - I_{j}^{F^{\times}} (\phi_1, \psi_1) \right| \nonumber \\
& = & \left|I_{j}^{F^{\times}} (\phi, \psi - \psi_1) \right| + \left|I_{j}^{F^{\times}} (\phi - \phi_1, \psi_1) \right| \nonumber \\
& \le & M_{3} \epsilon + M_4 \epsilon
\end{eqnarray}
Using Inequalities \ref{eqn:6.3}, \ref{eqn:6.4} and \ref{eqn:6.5}, we obtain the following estimate.
\begin{eqnarray}
\label{eqn:6.6}
\left | I_{j}^{F^{\times}} (\phi, \psi) - I (\phi) I (\psi) \right| & \le &   M_3 \epsilon + M_4 \epsilon + \left | I_{j}^{F^{\times}} (\phi_1, \psi_1) - I (\phi_1) I (\psi_1) \right| + M_1 \epsilon + M_2 \epsilon  
\end{eqnarray}
The proof of the theorem now follows by applying the hypothesis (i.e), the convergence mentioned in Equation \ref{eqn:mixcross} to the Inequality \ref{eqn:6.6}.
\end{proof}

\begin{remark}
One can observe that the conclusion of Theorem \ref{thm:alg} holds for product of two distinct holomorphic correspondences as well with an analogous hypothesis.
\end{remark}

Equipped with all the tools required, we are now ready to state and prove the main theorem of this section.

\begin{theorem}
\label{thm:main}
Let $F$ be a holomorphic correspondence on $X$ with topological degree $d$ and $\mu$ be an $F^*$ invariant probability measure. Then the following are equivalent:
\begin{enumerate}
    \item $F$ is weakly mixing with respect to $\mu$.
    \item  $F^{\times}$ is ergodic with respect to $\mu \times \mu$.
    \item $F^{\times}$ is weakly mixing with respect to $\mu \times \mu$.
\end{enumerate}
\end{theorem}

\begin{proof}
\textbf{1 $\implies$ 3} \ Assume that $F$ is weakly mixing with respect to $\mu$. By virtue of Theorem \ref{thm:alg}, it is enough to prove that for all $\phi \in P_{1}$ and $\psi \in P_{2}$,
\begin{equation}
\label{eqn:5}
 \lim \limits_{n \, \to \, \infty} \frac{1}{n} \sum \limits_{j=0}^{n-1}\left|I_{j}^{F^{\times}} (\phi, \psi) - I (\phi) I (\psi) \right|= 0  
\end{equation}
Let $\phi = \phi_1 * \phi_2$, where $\phi_1, \phi_2 \in \mathcal{C} (X, \mathbb{R})$ and let $\psi = \chi_{\raisebox{-.7ex}{$\scriptstyle B_1$}} * \chi_{\raisebox{-.7ex}{$\scriptstyle B_2$}}$, where $B_1, B_2 \in \mathscr{B} (X)$. Since $F$ is weakly mixing with respect to $\mu$, by Lemma \ref{lem:equi}, there exists density zero sets $D_1$ and $D_2$ such that
\begin{equation}
\label{eqn:3}
\lim \limits_{D_1 \, \notni \, n \, \to \, \infty} I_{n}^{F} (\phi_1, \chi_{\raisebox{-.7ex}{$\scriptstyle B_1$}}) = I(\phi_1) I(\chi_{\raisebox{-.7ex}{$\scriptstyle B_1$}})
\quad \quad \lim \limits_{D_2 \, \notni \, n \, \to \, \infty} I_{n}^{F} (\phi_2, \chi_{\raisebox{-.7ex}{$\scriptstyle B_2$}}) = I(\phi_2) I(\chi_{\raisebox{-.7ex}{$\scriptstyle B_2$}}).
\end{equation}
By Fubini's theorem we have, for all $n \in \mathbb{N}$,
\begin{equation}
\label{eqn:folland_1}
I_{n}^{F^{\times}} (\phi, \psi)  =  \bigintsss \limits_{X}  \left(\dfrac{1}{d^n} \displaystyle \sideset{}{'} \sum \limits_{w_1 \, \in \, \left(F^{n} \right)^{\dagger} (z_1)} \phi_1 (w_1) \right) \chi_{\raisebox{-.7ex}{$\scriptstyle B_1$}}(z_1) \, \mathrm{d}\mu \bigintsss \limits_{X}  \left(\dfrac{1}{d^n} \displaystyle \sideset{}{'} \sum \limits_{w_2 \, \in \, \left(F^{n} \right)^{\dagger} (z_2)} \phi_2 (w_2) \right) \chi_{\raisebox{-.7ex}{$\scriptstyle B_2$}}(z_2) \, \mathrm{d}\mu 
\end{equation}
Further, we also have,
\begin{equation}
\label{eqn:folland_2}
I (\phi)  =  \int \limits_{X} \phi_1(z_1) \, \mathrm{d}\mu \int \limits_{X} \phi_2(z_2) \, \mathrm{d}\mu \hspace{2mm} \text{and} \hspace{2mm} I (\psi) = \int \limits_{X} \chi_{B_1}(z_1) \, \mathrm{d}\mu \int \limits_{X} \chi_{B_2}(z_2) \, \mathrm{d}\mu
\end{equation}
Now, combining Equations \ref{eqn:3}, \ref{eqn:folland_1} and \ref{eqn:folland_2} gives
\begin{equation}
\label{eqn:densitydouble}
\lim \limits_{D_1\cup D_2 \, \notni \, n \, \to \, \infty} I_{n}^{F^{\times}} (\phi, \psi) = I (\phi) I (\psi). 
\end{equation}
Applying Proposition \ref{prop:wal20} to the sequence $a_n = I_{n}^{F^{\times}} (\phi, \psi) - I (\phi) I (\psi)$ and using Equation \ref{eqn:densitydouble} gives Equation \ref{eqn:5}, which proves that $F^{\times}$ is weakly mixing with respect to $\mu \times \mu$. 

\textbf{3 $\implies$ 2} \ This implication follows from Theorem \ref{thm:equiv}.

\textbf{2 $\implies$ 1} \ Assume that $F^{\times}$ is ergodic with respect to $\mu \times \mu$. Let $f \in \mathcal{C} (X, \mathbb{R})$ and $g \in L^{1} (\mu)$. Let $\pi_1: X \times X \to X$ be given by $\pi_{1} (z, w) = z$. Choosing $\phi=f\circ\pi_1$ and $\psi=g\circ\pi_1$ and applying Theorem \ref{thm:ergequiv}, we get $\lim \limits_{n \, \to \, \infty} \dfrac{1}{n} \displaystyle \sum \limits_{j \, =\, 0}^{n - 1} I_{j}^{F^{\times}} (f \circ \pi_1, g \circ \pi_1) = I (f \circ \pi_1) I (g \circ \pi_1).$ We observe that the above integral essentially boils down to an integral in single variable and as a result we get,
\begin{equation}
\label{eqn:firstpart}
\lim \limits_{n \, \to \, \infty} \frac{1}{n} \sum \limits_{j \,=\, 0}^{n - 1} I_{j}^{F} (f,g) = I(f)I(g)
\end{equation}

Similarly, by using Theorem \ref{thm:ergequiv} for the choice $\phi = f*f$ and $\psi = g*g$, we get
\begin{equation*}
    \lim \limits_{n \, \to \, \infty} \dfrac{1}{n} \displaystyle \sum \limits_{j \, =\, 0}^{n - 1} I_{j}^{F^{\times}} (f * f , g * g) = I (f * f) I (g * g),
\end{equation*}

which upon simplification results in      
\begin{equation}
\label{eqn:2}
\lim \limits_{n \, \to \, \infty}\frac{1}{n} \sum \limits_{j=0}^{n-1} \left (I_{j}^{F} (f,g) \right)^{2} = (I(f))^{2} (I(g))^{2}.
\end{equation}
By using Equations \ref{eqn:firstpart} and \ref{eqn:2}, we see that $\lim \limits_{n \, \to \, \infty} \dfrac{1}{n} \displaystyle \sum \limits_{j \, = \, 0}^{n-1} \left| I_{n}^{F} (f, g) - I(f) I(g) \right|^{2} = 0.$ At this stage, invoking Lemma \ref{lem:equi} shows that $F$ is weakly mixing with respect to $\mu$.
\end{proof}

\nocite{*}

\printbibliography[heading=bibintoc,title=References]
\thispagestyle{empty}

\end{document}